\documentclass[11pt]{article}
\usepackage{xr}
\usepackage{amssymb}
\usepackage{amsfonts}
\usepackage{amsmath}
\usepackage{theorem}
\usepackage{graphicx}
\usepackage{xcolor}
\usepackage[longnamesfirst]{natbib}

\newtheorem{theorem}{Theorem}

\newtheorem{condition}{Condition}

\newtheorem{lemma}{Lemma}

\newenvironment{proof}[1][Proof]{\textbf{#1.} }{\ \rule{0.5em}{0.5em}}
\allowdisplaybreaks

\begin{document}
\title{On Uniform Confidence Intervals for the Tail Index and the Extreme Quantile}
\author{Yuya Sasaki\footnote{Y. Sasaki: VU Station B \#351819, 2301 Vanderbilt Place, Nashville, TN 37235-1819, USA.}\\Vanderbilt University \and Yulong Wang\footnote{Y. Wang: 110 Eggers Hall, Syracuse, NY 13244-1020, USA.}\\Syracuse University}
\date{October 10, 2022}
\maketitle

\begin{abstract}\setlength{\baselineskip}{8.1mm}
This paper presents two results concerning uniform confidence intervals for the tail index and the extreme quantile.
First, we show that it is impossible to construct a length-optimal confidence interval satisfying the correct uniform coverage over a local non-parametric family of tail distributions.
Second, in light of the impossibility result, we construct honest confidence intervals that are uniformly valid by incorporating the worst-case bias in the local non-parametric family. 
The proposed method is applied to simulated data and a real data set of National Vital Statistics from National Center for Health Statistics. 
\end{abstract}

\newpage
\section{Introduction}
Suppose that one is interested in constructing a confidence interval (CI) for the true tail index $\xi_0 \in \mathbb{R}^{+}$ of a distribution. 
To define this parameter, assume that the distribution function (d.f.), denoted by $F$, has its well-defined density function $f$ and satisfies the well-known von Mises condition: 
\begin{equation}\label{eq:von_mises}
g\left( x\right) \equiv \frac{1-F\left( x\right) }{xf\left( x\right) } \rightarrow \xi
\end{equation}
as $x\rightarrow \infty $, where $\xi$ is called the tail index.
Point-wise CIs for $\xi_0$ have been proposed by a number of papers, including \citet{cheng2001confidence}, \citet{lu2002likelihood}, \citet{peng2006new}, and \citet{haeusler2007}.
More recently, \citet{carpentier2014} develop an adaptive CI for $\xi_0$ that is uniformly valid over a parametric family of tail distributions indexed by the second-order parameter.

Since the seminal paper by \citet{hill1975}, the literature has investigated numerous estimators for $\xi_0$ as well as their asymptotic properties. 
See the recent reviews by \citet{GomesGuillou2015review} and \citet{Fedotenkov2020review}. 
\citet{drees2001minimax} obtains the worst-case optimal convergence rate, i.e., the min-max bound, in estimation over a local non-parametric family of tail distributions.
\citet{CarpentierKim2014} construct adaptive and minimax optimal estimator over the parametric family of second-order Pareto-type distributions. 
Motivated by these results, a natural question is whether a length-optimal CI for $\xi_0$ can achieve uniformly correct coverage probabilities for $\xi_0$ over the non-parametric family of tail distributions.
To our best knowledge, the existing literature has not answered this important question, while such a problem has been investigated in other important contexts in statistics, e.g., for CIs of non-parametric densities \citep[e.g.,][]{low1997,HoffmannNickl2011,BullNickl2013,Carpentier2013}, non-parametric regressions \citep[e.g.,][]{Li1989, genovese2008adaptive}, and high-dimensional regression models \citep[e.g.,][]{Geer2014,WuWangFu2021} to list but a very few.

In this paper, we first show that it is in fact impossible to construct a length-optimal CI for the true tail index $\xi_0$ satisfying the uniformly correct coverage over the local non-parametric family considered by \citet{drees2001minimax}. 
Specifically, any CI enjoying the uniform coverage property can be no shorter than the worst-case bias over the non-parametric family up to a constant.
This negative result is analogous to those of \citet[][]{low1997} and \citet[][]{genovese2008adaptive} presented in the contexts of non-parametric densities/regressions, but is novel in the context of the tail index.

Second, we derive the asymptotic distribution of Hill's estimator uniformly over the local non-parametric family of tail distributions.
Given the above impossibility result, it is imperative to account for a possible range of bias over the non-parametric family.
Hence, we construct an honest CI for the tail index that is locally uniformly valid by incorporating the worse-case bias over the local non-parametric family, as well as influences from a sample on asymptotic randomness.
We also demonstrate that this honest CI for the tail index extends to that for extreme quantiles.

Simulation studies support our theoretical predictions.
While the na\"ive length-optimal CI not accounting for a possible range of bias over the local non-parametric family suffers from severe under-coverage overall, our proposed CIs on the other hand achieve correct coverage.
We apply the proposed method to National Vital Statistics from National Center for Health Statistics, and construct CIs for quantiles of extremely low infant birth weights across a variety of mothers' demographic characteristics and maternal behaviors. 
We find that, even after accounting for a possible range of bias over the local non-parametric family, having no prenatal visit during pregnancy remains a strong risk factor for low infant birth weight.

{\bf Organization:}
Section \ref{sec:setup} introduces the setup, notations, and definitions.
Section \ref{sec:impossibility} establishes the impossibility result.
Section \ref{sec:ci} proposes a uniformly valid CI.
Section \ref{eq:quantile} presents an application to extreme quantiles.
Sections \ref{sec:simulation} and \ref{sec:real_data} illustrate simulation studies and real data analyses, respectively.
Section \ref{sec:summary} summarizes the paper.
Mathematical proofs are collected in the Appendix.

\section{Setup, Notations, and Definitions}\label{sec:setup}
\subsection{Non-parametric Families in the Tail}\label{sec:nonparametric}

Any distribution function $F$ satisfying  the von Mises condition \eqref{eq:von_mises} can be equivalently characterized in the right tail in terms of its inverse $F^{-1}$ by 
\begin{equation*}
F^{-1}\left( 1-t\right) =ct^{-\xi }\exp \left( \int_{t}^{1} \frac{\eta\left(
s\right)}{s}ds\right)
\end{equation*}
for some $c>0$ and $\eta \left( s\right) =g\left( F^{-1}\left( 1-s\right)\right) -\xi $, which satisfies that $\eta(s)$ tends to zero as $s\rightarrow 0$. The standard Pareto
distribution falls in this family as a trivial special case in which $\eta$ is the zero function and $g\left(x\right) =\xi$ for all $x$.

To maintain a non-parametric setup in statistical inference about the true tail idex $\xi_0$, we follow \citet{drees2001minimax} and consider the following family of d.f.'s: 
\begin{equation*}
\mathcal{F}\left( \xi _{0},c,\varepsilon ,u\right) 
\equiv 
\left\{ F\text{ is a d.f. }\left\vert 
\begin{array}{c}
\left. F^{-1}\left( 1-t\right) =ct^{-\xi }\exp \left(\int_{t}^{1}\frac{\eta \left( s\right) }{s}ds\right), \right. \\ 
\left. \left\vert \xi -\xi _{0}\right\vert \leq \varepsilon ,\left\vert \eta \left( t\right) \right\vert \leq u\left( t\right) ,t\in (0,1]\right.
\end{array}
\right. \right\}
\end{equation*}
where $\xi _{0}>\varepsilon >0,c>0$ and $u\left( t\right) =At^{\rho }$ for some constants $A,\rho >0$.

Two remarks are in order about this family $\mathcal{F}\left( \xi _{0},c,\varepsilon ,u\right)$. 
First, let $F_{0}$ denote the Pareto distribution function with true tail index $\xi _{0}>0$, i.e., $F_{0}^{-1}\left( 1-t\right) =t^{-\xi _{0}}$. 
This $F_{0}$ is the center of localization by the family $\mathcal{F}\left( \xi _{0},c,\varepsilon,u\right) $. 
The factor $\exp \left( \int_{t}^{1}\frac{\eta \left( s\right) }{s}ds\right) $ represents a deviation from this center. 
If we set $\eta=u$, then the model essentially becomes parametric, since the deviation is fully determined by $u\left( t\right) =At^{\rho }$ and hence by the constants, $A$ and $\rho $. 
In this parametric setup, \citet{hallwelsh1985} establish the optimal uniform rates of convergence over a family of d.f.'s
with densities of the form $f\left( x\right) =cx^{-\left( 1/\xi +1\right) }\left( 1+r\left( x\right) \right)$, where $\left\vert r\left(x\right) \right\vert \leq Ax^{-\rho /\xi }$. 
In contrast, we consider a non-parametric family like $\mathcal{F}\left( \xi _{0},c,\varepsilon ,u\right) $, 
in which the function $u\left( \cdot \right) $ serves merely as an upper bound for deviations from the center. 
We will focus on the classic estimator by \citet{hill1975}. 
Since Hill's estimator is scale invariant, we hereafter assume $c=1$ without loss of generality and for succinct writing.

Second, by the fact that $\int_{t}^{1}\frac{\eta \left( 0\right) }{s}ds=-\eta \left( 0\right) \log t$, we can rewrite the quantile characterization $F^{-1}\left( 1-t\right) =t^{-\xi_0 }\exp \left(\int_{t}^{1} \frac{\eta \left( s\right) }{s}ds\right)$ of an element $F \in \mathcal{F}\left( \xi _{0},c,\varepsilon ,u\right)$ as 
\begin{equation}  \label{eq:rewrite_deviation}
t^{-\xi _{0}}\exp \left( \int_{t}^{1}\frac{\eta \left( s\right) }{s}ds\right) = t^{-\xi _{0}-\eta \left( 0\right) }\exp \left( \int_{t}^{1}\frac{\eta \left( s\right) -\eta \left( 0\right) }{s}ds\right).
\end{equation}
To construct a uniformly valid inference, we now again follow \citet{drees2001minimax} and consider a sequence of families of data-generating processes with tail index converging to $\xi _{0}$ at a rate $d_{n}:=u\left( t_{n}\right) \rightarrow 0$ for some $t_{n}\rightarrow 0$ as $n\rightarrow \infty $. 
Specifically, we consider a sequence $t_n$ satisfying
\begin{equation}
\label{eq:tn}
\lim_{n\rightarrow \infty}u(t_n)(nt_n)^{1/2} = 1. 
\end{equation}
This sequence essentially entails the optimal choice of the tuning parameter \citep[][p.77]{de2006extreme}, which we introduce later.
Now, we obtain a drifting sequence of local families consisting of $(n,h)$-indexed elements $F_{n,h}$ whose quantile functions satisfy 
\begin{align}
F_{n,h}^{-1}\left( 1-t\right) \equiv &t^{-\xi _{0}}\exp \left( \int_{t}^{1}\frac{d_{n}h\left( t_{n}^{-1}v\right) }{v}dv\right)  \label{eq:local_alternative}
\\
=&t^{-\left( \xi _{0}+d_{n}h\left( 0\right) \right) }\exp \left(\int_{t}^{1}\frac{d_{n}\left( h\left( {t_{n}^{-1}}v\right) -h\left( 0\right)\right) }{v}dv\right) ,t\in (0,1],  \notag
\end{align}
where the second equality is due to \eqref{eq:rewrite_deviation}. 
The corresponding tail index is given by $\xi _{n,h}=\xi _{0}+d_{n}h\left( 0\right) $, and $h\left( \cdot\right) -h\left( 0\right) $ characterizes the local deviation from the standard Pareto distribution.

Given the above reparametrization, we translate the local non-parametric family $\mathcal{F}\left( \xi _{0},c,\varepsilon ,u\right)$ of d.f.'s into that of $h$. 
Setting the upper bound to $u\left( s\right) =As^{\rho }$, we consider the family 
\begin{equation*}
\mathcal{H}\left( A,\rho \right) \equiv \left\{ h\in L_{2}[0,1]\left\vert 
\begin{array}{c}
\sup_{s>0}\left\vert h\left( s\right) \right\vert <\infty , \\ 
\left\vert h\left( s\right) -h\left( 0\right) \right\vert \leq As^{\rho },s>0
\end{array}
\right. \right\},
\end{equation*}
which contains all square integrable functions $h\left( \cdot \right) $ that are uniformly bounded and satisfy the bound $\left\vert h\left( s\right) -h\left( 0\right) \right\vert \leq As^{\rho }$. 
The family $\mathcal{H} \left( A,\rho \right) $ induces a local counterpart of the non-parametric family $\mathcal{F}\left( \xi _{0},1,\varepsilon ,u\right)$, namely 
\begin{align*}
&\mathcal{F}_n(\xi_0,\mathcal{H}(A,\rho),u) = \\
&\left\{ F_{n,h}\text{ is a d.f. } \left\vert 
\begin{array}{c}
\left. F_{n,h}^{-1}\left( 1-t\right) =t^{-(\xi_0+d_n h(0))}\exp \left(\int_{t}^{1}\frac{\eta(s)-\eta(0)}{v}dv\right),
\right. \\ 
\left. \eta(s) = d_n h(t^{-1}_ns), d_n=u(t_n), h \in \mathcal{H}(A,\rho) \right.
\end{array}
\right. \right\}
\end{align*}
indexed by $t^{-1}_n \rightarrow \infty$ as a function of the sample size $n$. 
Given the local reparametrization introduced above, $h(\cdot)$ now represents a deviation function and the associated tail index is $\xi_{n,h} = \xi_0 + d_n h(0)$.

\subsection{Hill's Estimator}\label{sec:hill}

Let $\{Y_{1},...,Y_{n}\}$ be a random sample from $F_{n,h}$, and let $Y_{n:n-j}$ denote the $(j+1)$-th largest order statistic in this sample. 
With these notations, Hill's estimator is defined by
\begin{equation*}
\hat{\xi}\left( n,k\right) =\frac{1}{k}\sum_{j=0}^{k-1}\left[ \log \left(Y_{n:n-j}\right) -\log \left( Y_{n:n-k}\right) \right] .
\end{equation*}
In practice, a researcher often implements this estimator for an interval of values of the tuning parameter $k$ to demonstrate \textit{ad hoc} robustness. 
This common practice can be formally accommodated by allowing for a sequence of intervals, i.e., $k=k_{n}=r\overline{k}_{n}$ for $r\in \lbrack \underline{r},1]\subset (0,1]$, similarly to \citet*{drees2000plot}.

Define the functional $T_{r}\left( z\right) $ by 
\begin{equation*}
T_{r}\left( z\right) \equiv r^{-1}\int_{0}^{r}\log \left( z\left( t\right) /z\left( r\right) \right) dt
\end{equation*}
for any measurable function $z:[0,1]\rightarrow \mathbb{R}$. 
If we substitute the quantile function of the standard Pareto distribution, that is, $z\left( t\right) =F_{0}^{-1}\left( 1-t\right) =t^{-\xi _{0}}$, then we have 
\begin{equation*}
T_{r}\left( t^{-\xi_0} \right) =r^{-1}\int_{0}^{r}\log \left( t^{-\xi_{0}}/r^{-\xi _{0}}\right) dt=\xi _{0},
\end{equation*}
identifying the true $\xi_0$
for any $r\in (0,1]$.
Define the tail empirical quantile function
\begin{equation*}
Q_{n,r\overline{k}_{n},F_{n,h}}=Y_{n:n-\lfloor r \overline{k}_{n}\rfloor}
\end{equation*}
for $r \in (0,1]$, where $\lfloor \cdot \rfloor$ denotes the smallest larger integer. 
With these auxiliary notations, as implied by Example 3.1 in \citet{drees1998estimate}, Hill's estimator $\hat{\xi}(n,r\overline{k}_{n})$ can be equivalently rewritten as $T_{r}\left( Q_{n,r\overline{k}_{n},F_{n,h}}\right) $.

\section{Asymptotic Impossibility}\label{sec:impossibility}

This section presents the first one of the two main theoretical results of this paper.
In light of the min-max result for estimation \citep[cf.][]{drees2001minimax},
it is natural that a researcher is interested in a length-optimal confidence interval satisfying the uniform coverage over a non-parametric family, such as the one introduced in Section \ref{sec:nonparametric}.
This section shows an impossibility of this objective. 

Specifically, we aim to establish that the length of any confidence interval $H\left( Y^{(n)}\right)$ $=$ $H\left( Y_{1},...,Y_{n}\right) $ that has a coverage of $\xi_0$ uniformly for all $h \in \mathcal{H}\left( A,\rho \right)$ is no shorter than the supremum seminorm of $h$ over $\mathcal{H}\left( A,\rho \right)$ up to a constant multiple.
This implies that we cannot find a length-optimal confidence interval which satisfies the uniform coverage without accounting for the worst-case bias $|T_r(F_{n,h}^{-1})-T_r(F_0^{-1})|$ over the non-parametric family $\mathcal{H}\left( A,\rho \right)$.
Such an impossibility result in spirit parallels the one about non-parametric density estimation established by \citet{low1997} and the one about non-parametric regression estimation established by \citet{genovese2008adaptive}, among others.

Define the modulus of continuity of $T_{r}\left( \cdot \right) $ by
\begin{equation}\label{eq:modulus}
\omega \left( \varepsilon ,F_{0}^{-1},n\right) \equiv \sup \left\{
\left\vert T_{r}\left( F_{n,h}^{-1}\right)
-T_{r}\left( F_{0}^{-1}\right) \right\vert \left\vert 
\begin{array}{c}
h\in \mathcal{H}\left( A,\rho \right) \text{ } \\ 
\left\vert \left\vert h-h\left( 0\right) \right\vert \right\vert
\leq \varepsilon
\end{array}
\right. \right\}. 
\end{equation}
This is the worst-case bias in absolute value.
Let $F_{0}^{n}$ and $F_{n,h}^{n}$ denote the joint distributions of $n$ i.i.d. draws from $F_{0}$ and $F_{n,h}$, respectively. 
Let $\mathbb{E}_{F_{0}^{n}}\left[ \cdot \right] $ denote the expectation with respect to the product measure $F_{0}^{n}$. 
The following theorem establishes the impossibility result. 
That is, the expected length of a uniformly valid confidence interval is no shorter than a constant multiplie of the modulus of continuity of $T_{r}\left( \cdot \right)$. 

\begin{theorem}[Impossibility]
\label{thm:impossibility} 
For $\varepsilon >0$, suppose that a confidence interval  $H\left(Y^{\left( n\right) }\right) $ for $\xi _{0}$ has coverage probabilities of at least $1-\beta $ uniformly for all $h\in \mathcal{H}\left( A,\rho \right) \cap \{h:\left\vert \left\vert h-h\left( 0\right) \right\vert \right\vert \leq \varepsilon \}$. 
Then, there exist $N$ and $C$ depending on $\varepsilon $ and $\xi _{0}$ only such that $1-\beta-C > 0$ and 
\begin{equation}
\mathbb{E}_{F_{0}^{n}}\left[ \mu \left( H\left( Y^{\left( n\right) }\right) \right) \right] \geq \left( 1-\beta -C\right) \omega \left( \varepsilon ,F_{0}^{-1},n\right)  \label{eq:lb}
\end{equation}
for all $n>N$.
\end{theorem}

As already mentioned above, this result is analogous to those established by \citet{low1997}, \citet{cai2004adaptation}, \citet{genovese2008adaptive}, and \citet{armstrong_kolesar_2018_Econometrica}, to list but a few, in other important contexts of statistics, but it is novel in the context of the tail index.
To understand the lower bound in \eqref{eq:lb}, we now derive a concrete expression for the element in the definition \eqref{eq:modulus} of the modulus of continuity $\omega \left( \varepsilon ,F_{0}^{-1},n\right)$. 
Note that
\begin{align*}
&T_{r}\left(F_{n,h}^{-1}\right)-T_r(F_0^{-1}) \\
=&r^{-1}\int_{0}^{r}\log \left( \frac{t^{-\xi _{0}}\exp \left( \int_{t}^{1}\frac{d_{n}h\left( {t_{n}^{-1}}v\right) }{v}dv\right) }{r^{-\xi _{0}}\exp\left(\int_{r}^{1}\frac{d_{n}h\left( {t_{n}^{-1}}v\right) }{v}dv\right) }\right)dt-\xi_{0} \\
=&r^{-1}\int_{0}^{r}\left[ -\xi _{0}\log \left( t/r\right)+d_{n}\left(\int_{t}^{1}\frac{h\left( {t_{n}^{-1}}v\right) }{v}dv-\int_{r}^{1}\frac{h\left({t_{n}^{-1}}v\right) }{v}dv\right) \right] dt-\xi _{0} \\
=&d_{n}r^{-1}\int_{0}^{r}\left( \int_{t}^{1}\frac{h\left( {t_{n}^{-1}}v\right) }{v}dv-\int_{r}^{1}\frac{h\left( {t_{n}^{-1}}v\right) }{v}dv\right)dt \\
=&d_{n}r^{-1}\int_{0}^{r}\left( \int_{t}^{r}\frac{h\left( {t_{n}^{-1}}v\right) }{v}dv\right) dt \\
=&d_{n} r^{-1}\int_{0}^{r}\left( \int_{t}^{r}\frac{h\left({t_{n}^{-1}}v\right)-h\left( 0\right) }{v}dv\right) dt + d_{n}h\left( 0\right).
\end{align*}
The first term in the last line characterizes the bias due to the deviation from the standard Pareto distribution, and the second term characterizes the asymptotic randomness. 
As a consequence, to obtain a feasible and uniformly valid confidence interval, we will set an upper bound for the first term and adjust the critical value based on the second term. 
We obtain such a uniform confidence interval in the following section.

\section{Uniform Confidence Interval}\label{sec:ci}

Given that $\rho>0$, it has been established that the optimal rate of convergence of Hill's estimator is $n^{-\rho/(2\rho+1)}$. 
See, for example, Remark 3.2.7 in \citet{de2006extreme}. 
Such an optimal rate entails non-negligible asymptotic bias as charaterized in Theorem \ref{thm:index} below. 
To achieve this rate, we let $\overline{k}_n \approx nt_n$, where $A \approx B$ means $\lim_{n\rightarrow \infty}A/B = 1$. 
Then, the restriction (\ref{eq:tn}) implies that $u(t_n) \approx A^{1/(2\rho+1)}n^{-\rho/(2\rho+1)}$ and $\overline{k}_n^{1/2}u(t_n) =  \overline{k}_n^{1/2}d_n \rightarrow 1$ as $n \rightarrow \infty$.
We formally summarize these conditions below.
\begin{condition}
As $n \rightarrow \infty$,  $ A^{1/(2\rho+1)}\overline{k}_n^{1/2}n^{-\rho/(2\rho+1)} \rightarrow 1$. 
\label{cond:k}
\end{condition}

As the second one of the two main theoretical results of this paper, the following theorem derives the asymptotic distribution of $\hat{\xi} \left( n,r\overline{k}_n\right) $ uniformly for all $r \in [\underline{r},1]$ and $h\in \mathcal{H}\left( A,\rho \right)$ by exploiting the features of the functional $T_{r}\left( \cdot \right) $ defined in Section \ref{sec:hill}. 
Let $\mathbb{P}^n_{n,h}$ denote the distribution of $n$ i.i.d. draws from the distribution $F_{n,h}$. 

\begin{theorem}[Uniform Asymptotic Distribution]
\label{thm:index}If Condition \ref{cond:k} is satisfied, then
\begin{equation*}
\overline{k}_n \left( \hat{\xi}\left( n,r\overline k_n\right) -\xi _{n,h} \right) =\xi _{0}  \mathbb{G}\left( r\right) + \mathbb{B}\left( r;h\right) + o_{\mathbb{P}^n_{n,h}}(1)
\end{equation*}
holds under random sampling from \eqref{eq:local_alternative} uniformly for all $r \in \lbrack \underline{r},1]$ and $h\in \mathcal{H}\left( A,\rho \right)$,
 where $\mathbb{G}$ and $\mathbb{B}$ are defined by
\begin{equation}
\mathbb{G}\left( r\right) \equiv r^{-1}\int_{0}^{r}\left( s^{-1}W\left( s\right) -r^{-1}W\left( r\right) \right) ds
\label{eq:G}
\end{equation}
and
\begin{equation}
\mathbb{B}\left( r;h\right) \equiv r^{-1}\int_{0}^{r}\left( \int_{s}^{r} \frac{h\left( v\right) -h\left( 0\right) }{v}dv\right) ds,  \label{eq:B}
\end{equation}
respectively.
\end{theorem}

A proof can be found in the Appendix. 
The convergence of Hill's estimator as a function of $r$ has been established in the literature \citep[e.g.,][]{resnick1997, drees2000plot}. 
In comparison, Theorem \ref{thm:index} here contributes to the literature by showing the asymptotic distribution uniformly over both $r \in [\underline{r}, 1]$ and the local non-parametric function class.  
The terms $\mathbb{G}\left( \cdot \right) $ and $\mathbb{B}\left( \cdot ;h\right) $ characterize the asymptotic randomness and bias, respectively. 
It follows that
\begin{equation}
\sqrt{r\overline{k}_{n}}\left( \hat{\xi}\left( n,r\overline{k}_{n}\right) -\xi_{n,h}\right) \Rightarrow \xi _{0} \sqrt{r} \mathbb{G}\left( r\right)  + \sqrt{r} \mathbb{B}\left( r;h\right) .
\label{eq:xi_limit}
\end{equation}
To conduct statistical inference based on \eqref{eq:xi_limit}, we need to compute the bound 
\begin{equation*}
\sup_{r\in \lbrack \underline{r},1]}\sup_{h\in \mathcal{H}\left( A,\rho\right) }  \sqrt{r} \mathbb{B}
\left(r;h\right)
\end{equation*}
of the bias. To this end, note that $\left\vert h\left(v\right) -h\left( 0\right) \right\vert \leq Av^{\rho }$ for all $h\in \mathcal{H}\left( A,\rho \right) $. 
Therefore, for any $h\in \mathcal{H}\left( A,\rho \right) $, 
\begin{align}
\left\vert \mathbb{B}\left( r;h\right) \right\vert =&\left\vert r^{-1}\int_{0}^{r}\left( \int_{s}^{r}\frac{h\left( v\right) -h\left(
0\right) }{v}dv\right) ds\right\vert  \notag \\
\leq &r^{-1}\int_{0}^{r}\left( \int_{s}^{r}\frac{\left\vert h\left( v\right) -h\left( 0\right) \right\vert }{v}dv\right) ds  \notag \\
\leq &r^{-1}\int_{0}^{r}\left( \int_{s}^{r}Av^{\rho -1}dv\right) ds  \notag
\\
=&A\frac{r^{\rho }}{1+\rho }.  \label{eq:B_bound}
\end{align}
This bound is tight and achieved when, for example, $h\left( v\right) \geq h\left( 0\right) \geq 0$.

With this tight bias bound taken into account in a similar spirit to \citet{armstrong2020simple}, a locally uniformly valid confidence interval for the tail index is given by
\begin{align}
\label{eq:honestCI}
H^{O}\left( n,r\overline k_n\right) 
=\left[ \hat{\xi} - \frac{\hat{\xi}\cdot q_{1-\beta /2}+\overline{B} \left( A,\rho \right) }{\sqrt{r\overline k_n} }  , 
          \hat{\xi} + \frac{\hat{\xi}\cdot q_{1-\beta /2}+\overline{B} \left( A,\rho \right) }{\sqrt{r\overline k_n} }  \right]
\end{align}
for $r \in [\underline{r},1]$, where $\hat{\xi}$ is short for $\hat{\xi}\left( n, r\overline k_n\right)$, 
\begin{equation*}
\overline{B}\left(A,\rho \right) \equiv \sup_{r \in (0,1]}  \sqrt{r} \frac{Ar^{\rho }}{ 1+\rho } = \frac{A}{1+\rho}
\end{equation*}
is the upper bound of the bias, and $q_{1-\beta /2}$ denotes the suitable quantile of $\sup_{[\underline{r},1]} \sqrt{r}\mathbb{G}(r)$ whose values can be found in Table \ref{tab:cv}.

\begin{table}[t]
\renewcommand{\arraystretch}{0.85} 
\centering
\begin{tabular}{lccccclcccc}
\hline
\underline{$r$}$ \hspace{0.2cm} \backslash \hspace{0.2cm} \beta$ &  & 0.10 & 0.05 & 0.01 & &  \underline{$r$}$ \hspace{0.2cm} \backslash \hspace{0.2cm} \beta$ &  & 0.10 & 0.05 & 0.01  \\ 
     \hline
     1 &  & 1.64 & 1.96 & 2.56 & & 1/4  &  & 2.41 & 2.71 & 3.27 \\ 
10/11& & 1.87 & 2.19 & 2.76 & & 1/5  &  & 2.46 & 2.74 & 3.34 \\ 
5/6  &  & 1.95 & 2.27 & 2.86 & & 1/10 &  & 2.58 & 2.85 & 3.44 \\ 
2/3  &  & 2.09 & 2.42 & 3.01 & & 1/20 &  & 2.67 & 2.92 & 3.51 \\ 
1/2  &  & 2.22 & 2.54 & 3.12 & & 1/50 &  & 2.75 & 3.01 & 3.57 \\ 
1/3  &  & 2.33 & 2.66 & 3.23 & & 1/100 & & 2.80 & 3.08 & 3.61 \\ 

\hline
\end{tabular}
\caption{The $1-\beta/2$ quantile of $\sup_{r\in \left[ \protect\underline{r},1\right] } \sqrt{r}\mathbb{G}\left( r\right)  $ computed based on 20,000 simulation draws. 
The Gaussian process is approximated with 50,000 steps.}
\label{tab:cv}
\end{table}

As a final remark, we discuss the choice of the higher-order parameters $\rho$ and $A$. 
They both depend on the underlying distribution and hence unknown.
While this feature appears to be a disadvantage of our method, the impossibility result in Theorem \ref{thm:impossibility} implies that this feature cannot be avoided, regardless of how we construct the interval. 
The existing literature has proposed several estimators of $\rho$ and $A$, whose consistency requires additional assumptions on the underlying function, and equivalently further restrictions on the class $\mathcal{H}$. 
See \cite{CarpentierKim2014,cheng2001confidence,haeusler2007} for some data-driven methods of choosing the higher-order parameters.   
The corresponding confidence intervals are no longer uniformly valid, but still point-wisely valid. 

As an alternative, we propose a rule-of-thumb choice of $\rho$ and $A$ and proceed with the proposed interval. 
In particular, if the underlying distribution is Student-t with $1/\xi_0$ degrees of freedom, we know that $\rho = 2\xi_0$. 
Furthermore, for the bias upper bound $\overline{B} = A/(1+\rho)$, we set $A = 0.1 \xi_0 (1+2\xi_0) \sqrt{\overline{k}_n}$, so that $\overline{B}/\sqrt{\overline{k}_n}$ is at most $10\%$ of the true tail index to be estimated.  
In practice, we replace $\xi_0$ with the estimator $\hat{\xi}(n,\overline{k}_n)$. 
This rule-of-thumb choice is reministic to the Silverman's choice of bandwidth in kernel density estimation, where the reference is the Gaussian distribution. 
We examine the performance of this rule by simulations in Section \ref{sec:simulation}. 

\section{Extreme Quantiles}\label{eq:quantile}

In this section, we apply Theorem \ref{thm:index} to uniform inference about extreme quantiles. To characterize the extremeness, we focus on the sequence of $1-p_{n}$ quantiles where $p_{n}\rightarrow 0$ as $n\rightarrow \infty $. 
Consider the extreme quantile estimator by \citet{weissman1978estimation}: 
\begin{equation*}
\hat{F}_{n,h}^{-1}\left( 1-p_{n}\right) =Y_{n,n-\lfloor r \overline k_n \rfloor }\left( \frac{r \overline k_{n}}{np_{n}} \right) ^{\hat{\xi}\left( n,r \overline k_n\right) },
\end{equation*}
where $\hat{\xi}\left( n,k\right) $ is Hill's estimator. 
Recall that the true quantile under the local drifting sequence is 
\begin{equation*}
F_{n,h}^{-1}\left( 1-t_{n}\right) =t^{-\xi _{0}}\exp \left( \int_{t}^{1}  \frac{d_{n}h\left( {t_{n}^{-1}}v\right) }{v}dv\right)
\end{equation*}
as in \eqref{eq:local_alternative}, where $t_{n}$ satisfies $\lim_{n\rightarrow \infty }\overline{k}_{n}/(nt_n)=1$ as in Condition \ref{cond:k}. 
We now aim to asymptotically approximate the distribution of 
\begin{equation*}
 \frac{\hat{F}_{n,h}^{-1}\left( 1-p_{n}\right) }{F_{n,h}^{-1} \left(1-p_{n}\right) }-1.
\end{equation*}
To this end, we state the following condition on the relation among $n$, $\overline k_n$ and $p_n$ \citep[e.g.,][Theorem 4.3.1]{de2006extreme}.

\begin{condition}
\label{cond:p}$np_{n}=o\left( \overline k_{n}\right) $ and $\log \left(
np_{n}\right)=o\left( \sqrt{\overline k_{n}}\right) $ as $n\rightarrow \infty $.
\end{condition}

\begin{theorem}[Extreme Quantiles]
\label{thm:quantile}Under Conditions \ref{cond:k} and \ref{cond:p} with random sampling, uniformly for all $r \in \lbrack \underline{r},1]$ and $h\in \mathcal{H}\left( A,\rho \right)$, 
\begin{align*}
 \frac{\sqrt{r\overline{k}_n}}{\log \left( r \overline k_n/\left( np_{n}\right) \right) }\left( \frac{\hat{F}_{n,h}^{-1}\left( 1-p_{n}\right)  }{F_{n,h}^{-1}\left( 1-p_{n}\right) }-1\right) -\sqrt{r\overline{k}_n}\left( \hat{ \xi}\left( n,r \overline{k}_n\right) -\xi _{n,h}\right) 
=o_{\mathbb{P}^n_{n,h}}(1).  
\end{align*}
\end{theorem}

Theorem \ref{thm:quantile} implies that the asymptotic distribution of the extreme quantile estimator is the same as that of the tail index. 
Therefore, a uniformly valid confidence interval can be constructed similarly. 
In particular, a robust confidence interval for $F^{-1}_{n,h} (1-p_n)$ with nominal $1-\beta $ uniform coverage probability accounting for the bias is constructed as
\begin{align}
I^O\left(n, k\right) = &\left[\hat{F}^{-1}_{n,h} \left( 1-p_{n}\right) \left\{1 - \frac{\log d(n,r\overline{k}_n)}{\sqrt{r\overline{k}_n}} \left( \hat{\xi} \cdot q_{1-\beta /2} + \overline{B}(A,\rho)\right) \right\},\right.
\notag\\
& \:\:                         \left.\hat{F}^{-1}_{n,h} \left( 1-p_{n}\right) \left\{1 + \frac{\log d(n,r\overline{k}_n)}{\sqrt{r\overline{k}_n}} \left( \hat{\xi} \cdot q_{1-\beta /2} + \overline{B}(A,\rho)\right) \right\} \right],
\label{eq:quantile_honestCI}
\end{align}
where again $\hat{\xi}$ is short for $\hat{\xi}(n, r\overline{k}_n)$,  $q_{1-\beta/2}$ denotes a suitable quantile that can be found in Table \ref{tab:cv}, and $d(n,r\overline{k}_n) = r\overline{k}_n/(np_n)$.

\section{Simulation Studies}\label{sec:simulation}

In this section, we use simulated data to evaluate finite-sample performance of our proposed confidence intervals in comparison with the na\"ive lenth-optimal confidence interval.
Sections \ref{sec:simu_index} and \ref{sec:simu_quantile} focus on inference about the tail index and extreme quantiles, respectively.

\subsection{Tail Index}\label{sec:simu_index}
The following simulation design is employed for our analysis.
We generate $n$ independent standard uniform random variables $U_i$ and construct the observations $Y_i = F^{-1}(1-U_i) $, where $F^{-1}(1-t)=t^{-\xi_0}\exp (\int^{1}_{t} s^{-1}\eta(s)ds)$. 
We set $\eta(t) = c t^{\rho} $, so that 
\begin{equation}
\label{eq:dgp}
F^{-1}(1-t) = t^{-\xi_0}\exp \left( \frac{c (1-t^{\rho})}{\rho}  \right),  
\end{equation}
where the constants, $c$ and $\rho$, characterize the scale and the shape, respectively, of the deviation from the Pareto distribution. 
For ease of comparisons, we set $\rho = 2\xi_0$, which corresponds to the Student-t distribution with $1/\xi_0$ degrees of freedom. 
Then, we set $c = c_0 \xi_0/(1+2\xi_0) $  for normalization, where we vary $c_0 \in \{0,0.5,1\}$ across sets of simulations.

To construct the optimal confidence interval, we need a choice of $\overline{k}_n$. 
We use the data-driven algorithm proposed by \citet{guillou2001diagnostic}, which we briefly summarize in Appendix \ref{sec:data_driven_k} for the convenience of readers.
This choice of $\overline{k}_n$ is of the optimal rate as established in their Theorem 2.
Specifically, we select $\overline{k}_n$ according to \eqref{k choice} in Appendix \ref{sec:data_driven_k} with the restriction that $\overline{k}_n \in [0.01n,0.99n]$.

We implement three confidence intervals for the purpose of comparisons. 
The first method is the na\"ive length-optimal confidence interval without accounting for a possible bias over the local non-parametric family, that is,
\begin{equation}
H^{N}(n,\overline{k}_n) = \left[\hat{\xi}(n,\overline{k}_n)\pm1.96\overline{k}^{-1/2}_n \hat{\xi}(n,\overline{k}_n)\right],
\label{eq:naiveCI}
\end{equation}
where $\overline{k}_n$ is selected according to the procedure described above.
Our impossibility result predicts that it fails to achieve a correct coverage uniformly over the non-parametric class encompassing our simulation design presented above.

The second one is $H^O(n,\overline{k}_n)$, i.e., our proposed confidence interval given in \eqref{eq:honestCI} with $r=1$, where $\overline{k}_n$ is selected according to the procedure described above.
The bias upper bound $\overline{B} = A/(1+\rho)$ is chosen following the rule-of-thumb choice described in Section \ref{sec:ci}. 

The third is based on $k$ snooping. 
Specifically, given $\overline{k}_n$ selected according to the procedure described above, now consider the range $[\underline{r} \overline{k}_n, \overline{k}_n]$ containing $m$ integers denoted by $\{k_1,\cdots, k_m\}$.
The $k$ snooping interval is constructed by
\begin{equation}
H^S(n,\underline{r}) = \bigcap_{j=1}^m H^O(n, k_j)
\label{eq:snoopingCI},
\end{equation} 
where $H^O(\cdot,\cdot)$ is defined in \eqref{eq:honestCI} and the lower bound $\underline{r}$ of $k$ snooping is set to $1/2$. 
For the bias upper bound  $\overline{B} = A/(1+\rho)$, we set $A = 0.1 \xi_0 (1+2\xi_0) \sqrt{r\overline{k}_n}$ and replace $\xi_0$ with the estimator $\hat{\xi}(n,k_j)$ for each $j \in \{1,...,m\}$.

\begin{table}
\renewcommand{\arraystretch}{0.85} 
\centering
\begin{tabular}{lccccccccccccccc}
\multicolumn{13}{c}{Tail Index: Coverage Probabilities}\\
\hline
$n$ & & \multicolumn{3}{c}{250} && \multicolumn{3}{c}{500} && \multicolumn{3}{c}{1000} \\
\cline{3-5}\cline{7-9}\cline{11-13}
$(\xi_0,c_0)$ &  & $H^N$ & $H^O$ & $H^S$ && $H^N$ & $H^O$ & $H^S$ && $H^N$ & $H^O$ & $H^S$\\ \hline
(1, 0)         &  & 0.92 & 0.99 & 0.98 && 0.92 & 0.99 & 0.99 && 0.91 & 0.99 & 0.99 \\ 
(1, 0.5)      &  & 0.88 & 0.98 & 0.97 && 0.81 & 0.99 & 0.98 && 0.73 & 0.99 & 0.99 \\ 
(1, 1)         &  & 0.77 & 0.98 & 0.97 && 0.64 & 0.98 & 0.98 && 0.65 & 0.99 & 0.99 \\  
(0.5, 0)      &  & 0.91 & 0.99 & 0.98 && 0.91 & 0.99 & 0.99 && 0.91 & 0.99 & 0.99 \\ 
(0.5, 0.5)   &  & 0.70 & 0.98 & 0.98 && 0.55 & 0.98 & 0.99 && 0.55 & 0.98 & 0.98 \\ 
(0.5, 1)      &  & 0.51 & 0.87 & 0.94 && 0.59 & 0.86 & 0.92 && 0.61 & 0.93 & 0.95 \\ 
\hline
\\
\multicolumn{13}{c}{Tail Index: Average Lengths}\\
\hline
$n$ & & \multicolumn{3}{c}{250} && \multicolumn{3}{c}{500} && \multicolumn{3}{c}{1000} \\
\cline{3-5}\cline{7-9}\cline{11-13}
$(\xi_0,c_0)$ &  & $H^N$ & $H^O$ & $H^S$ && $H^N$ & $H^O$ & $H^S$ && $H^N$ & $H^O$ & $H^S$\\ \hline
(1, 0)         &  & 0.29 & 0.49 & 0.53 && 0.20 & 0.40 & 0.43 && 0.14 & 0.34 & 0.36 \\ 
(1, 0.5)      &  & 0.30 & 0.50 & 0.54 && 0.21 & 0.42 & 0.43 && 0.16 & 0.36 & 0.37 \\ 
(1, 1)         &  & 0.31 & 0.52 & 0.54 && 0.23 & 0.44 & 0.45 && 0.18 & 0.38 & 0.39 \\ 
(0.5, 0)      &  & 0.14 & 0.24 & 0.26 && 0.10 & 0.20 & 0.22 && 0.07 & 0.17 & 0.18 \\ 
(0.5, 0.5)   &  & 0.16 & 0.27 & 0.28 && 0.12 & 0.22 & 0.23 && 0.09 & 0.20 & 0.20 \\ 
(0.5, 1)      &  & 0.18 & 0.29 & 0.31 && 0.15 & 0.25 & 0.26 && 0.12 & 0.22 & 0.23 \\ 
\hline
\end{tabular}
\caption{The coverage probabilities (top panel) and the average lengths (bottom panel) of the 95\% confidence intervals for $\xi_0$. $H^N$ stands for the na\"ive confidence interval defined in \eqref{eq:naiveCI}. $H^O$ stands for our proposed confidence interval defined in \eqref{eq:honestCI} with $r=1$. $H^S$ stands for our proposed confidence interval with snooping defined in \eqref{eq:snoopingCI}. The results are based on 5000 simulation draws. }
\label{tab:tail_index}
\end{table}

From Theorem \ref{thm:index}, we expect that both $H^O(n,\overline{k}_n)$ and $H^S(n,\underline{r})$ will deliver asymptotically correct (uniform) coverages, whereas $H^{N}(n,\overline{k}_n)$ will not.
Table \ref{tab:tail_index} presents the coverage probabilities (top panel) and the average lengths (bottom panel) of the 95\% confidence intervals based on 5000 Monte Carlo iterations. 
Key findings can be summarized as follows.
First, both the intervals $H^O$ and $H^S$ have the correct coverage probability for most of the distributions consistently with our theory. 
When $\xi_0 = 0.5$ and $c_0 = 1$, the deviation of the tail distribution away from the Pareto distribution is the most severe. 
In this case, $H^O$ suffers from some undercoverage when $n$ is small (e.g., $n=250$ and $500$), but achieves more satisfactory coverage as $n$ becomes large (e.g., $n=1000$).
Second, the coverage by $H^N$ is inadequate throughout, and even when the deviation from the Pareto distribution is relatively small. 
Furthermore, the extent of the undercoverage by $H^N$ tends to exacerbate as the sample size $n$ increases.
Finally, the lengths of $H^S$ are slightly larger than those of $H^O$ when $n$ is 250, but they become almost identical when $n$ gets larger.  
From these findings, we prefer $H^S$ and $H^O$ to $H^N$.
As the sample size becomes larger, $H^S$ and $H^O$ become equally preferable, while $H^N$ consistently underperforms.x

\subsection{Extreme Quantiles}\label{sec:simu_quantile}

We now turn to extreme quantiles.
The data generating process continues to be the same as in the previous subsection. 
The object of interest is the 99\% quantile so that $p_n = 0.01$. 
Similarly to the previous subsection, we again compare three confidence intervals. 

The first one is the na\"ive confidence interval without accounting for a possible bias in the non-parametric family, that is,
\begin{align}
I^N\left(n, \overline{k}_n\right)
= &\left[\hat{Q}\left( 1-p_{n}\right) \left(1 - \frac{1.96}{\sqrt{\overline{k}_n}} \hat{\xi}(n,\overline{k}_n) \right),\right.
\:\: \left.\hat{Q}\left( 1-p_{n}\right)\left(1+\frac{1.96}{\sqrt{\overline{k}_n}}\hat{\xi}(n, \overline{k}_n) \right)\right].
\label{eq:quantile_naiveCI}
\end{align}
The second one is our proposed confidence interval, that is, $I^O(n,\overline{k}_n)$ in \eqref{eq:quantile_honestCI} with $\underline{r} = 1$.
The third one is our proposed confidence interval with $k$ snooping interval, that is,
\begin{equation}
I^S(n,\underline{r}) = \bigcap_{j=1}^m I^O(n, k_j)
\label{eq:quantile_snoopingCI},
\end{equation} 
where $\{ k_1,\cdots, k_m\}$ are constructed in the same way as in \eqref{eq:snoopingCI}, and $I^O(n, k_j)$ is defined as in \eqref{eq:quantile_honestCI}, and we set $\underline{r} = 1/2$.

\begin{table}
\renewcommand{\arraystretch}{0.85} 
\centering
\begin{tabular}{lccccccccccccccc}
\multicolumn{13}{c}{Extreme Quantile: Coverage Probabilities}\\
\hline
$n$ & & \multicolumn{3}{c}{250} && \multicolumn{3}{c}{500} && \multicolumn{3}{c}{1000} \\
\cline{3-5}\cline{7-9}\cline{11-13} 
$(\xi_0,c_0)$ &  & $I^N$ & $I^O$ & $I^S$ && $I^N$ & $I^O$ & $I^S$ && $I^N$ & $I^O$ & $I^S$\\ \hline
(1, 0)         &  & 0.90 & 0.96 & 0.94 && 0.92 & 0.98 & 0.97 && 0.93 & 0.99 & 0.99 \\ 
(1, 0.5)      &  & 0.92 & 0.96 & 0.93 && 0.93 & 0.98 & 0.97 && 0.89 & 0.99 & 0.98 \\ 
(1, 1)         &  & 0.92 & 0.95 & 0.93 && 0.89 & 0.97 & 0.95 && 0.83 & 0.99 & 0.97 \\ 
(0.5, 0)      &  & 0.92 & 0.98 & 0.97 && 0.93 & 0.99 & 0.98 && 0.93 & 0.99 & 0.99 \\ 
(0.5, 0.5)   &  & 0.91 & 0.97 & 0.95 && 0.86 & 0.98 & 0.96 && 0.78 & 0.99 & 0.98 \\ 
(0.5, 1)      &  & 0.86 & 0.96 & 0.94 && 0.81 & 0.97 & 0.96 && 0.85 & 0.98 & 0.97 \\ 
\hline
\\
\multicolumn{13}{c}{Extreme Quantile: Average Lengths}\\
\hline
$n$ & & \multicolumn{3}{c}{250} && \multicolumn{3}{c}{500} && \multicolumn{3}{c}{1000} \\
\cline{3-5}\cline{7-9}\cline{11-13}
$(\xi_0,c_0)$ &  & $I^N$ & $I^O$ & $I^S$ && $I^N$ & $I^O$ & $I^S$ && $I^N$ & $I^O$ & $I^S$\\ \hline
(1, 0)         &  & 136 & 232 & 240 && 91 & 183 & 189 && 62 & 152 & 156 \\ 
(1, 0.5)      &  & 170 & 291 & 277 && 110 & 225 & 213 && 76 & 185 & 174 \\ 
(1, 1)         &  & 199 & 342 & 302 && 132 & 263 & 235 && 90 & 209 & 189 \\ 
(0.5, 0)      &  & 6.3 & 10.9 & 11.6 && 4.4 & 8.9 & 9.3 && 3.0 & 7.5 & 7.7 \\ 
(0.5, 0.5)   &  & 8.3 & 14.4 & 14.3 && 5.8 & 11.5 & 11.3 && 4.2 & 9.4 & 9.1 \\ 
(0.5, 1)      &  & 10.9 & 18.3 & 17.6 && 7.5 & 13.7 & 13.3 && 5.3 & 10.6 & 10.4 \\ 
\hline
\end{tabular}
\caption{The coverage probabilities (top panel) and the average lengths (bottom panel) of the 95\% confidence intervals for the 99\% quantile. 
$I^N$ stands for the na\"ive confidence interval defined in \eqref{eq:quantile_naiveCI}. 
$I^O$ stands for our proposed confidence interval defined in \eqref{eq:quantile_honestCI} with $r=1$. 
$I^S$ stands for our proposed confidence interval with snooping defined in \eqref{eq:quantile_snoopingCI}. 
The results are based on 5000 simulation draws. }
\label{tab:quantile}
\end{table}

Table \ref{tab:quantile} presents the coverage probabilities (top panel) and the average lengths (bottom panel) of the 95\% confidence intervals based on 1000 Monte Carlo iterations.
The findings are similar to those in Table \ref{tab:tail_index}. 
In particular, both the intervals $I^O$ and $I^S$ lead to correct coverage probabilities for all distributions, while $I^N$ suffers from undercoverage in general.
Regarding the lengths, $I^O$ and $I^S$ are both longer than $I^N$ to allow for the correct coverage.
Furthermore, when $\xi_0$ is 0.5, $I^O$ has approximately the same lengths as $I^S$. 
When $\xi_0=1$, the lengths of $I^S$ are shorter than those of $I^O$, especially when the model largely deviates away from the Pareto distribution. 
This is because the true quantile is substantially larger so that the effects of adapting the critical value become more significant.
We prefer $I^O$ and $I^S$ to $I^N$ from these observations.

\section{Real Data Analysis}\label{sec:real_data}

This section illustrates an application of the proposed method to an analysis of extremely low infant birth weights.
Their relations with mothers' demographic characteristics and maternal behaviors address important research questions. 
We use detailed natality data (Vital Statistics) published by the National Center for Health Statistics, which has been used by prior studies including \cite{Abrevaya01} and many others. 
Our sample consists of repeated cross sections from 1989 to 2002.
Using the data from each of these years, we construct 95\% confidence intervals of the tail index in the left tail and the first percentile following the same computational procedure as the one taken in Section \ref{sec:simulation}.
Details of our implementation with the current empirical data set are as follows. 

We follow previous studies \citep[e.g.,][]{Abrevaya01} to choose the variables for mothers' demographic characteristics and maternal behaviors. 
The variable of our interest is the infant birth weight measured in kilograms.
For the purpose of comparison, we set a benchmark subsample in which the infant is a boy and the mother is younger than the median age in the full sample, is white and married, has levels of education lower than a high school degree, had her first prenatal visit in the first trimester (natal1), and did not smoke during the pregnancy. 
In addition to this benchmark subsample (benchmark), we also consider seven alternative subsamples corresponding to one and only one of the following scenarios: 
the mother has at least a high school diploma (high school); 
the infant is a girl (girl); 
the mother is unmarried (unmarried); 
the mother is black (black); 
the mother did not have prenatal visit during pregnancy (no pre-visit); 
the mother smokes ten cigarettes per day on average (smoke) and 
the mother's age is above the median age in the full sample (older).

For each of these subsamples, we construct the 95\% confidence intervals $H^N$, $H^O$, and $H^S$ for the tail index in the left tail in the same way as in Section \ref{sec:simu_index}, and the 95\% confidence intervals $I^N$, $I^O$, and $I^S$ for the first percentile as in Section \ref{sec:simu_quantile}.
Since we are interested in the left tail (extremely low birth weights), we consider only the birth weight, denoted $B_i$, that is less than some cutoff value $T$, and take $Y_i = T - B_i$ as the input in our computational procedure for inference.
We choose $T = 4$ based on the prior findings in \cite{Abrevaya01}.
Namely, \cite{Abrevaya01} finds that the relationship between the infant birth weight and mother's demographics change substantially at the 90th percentile of birth weight in the full sample, which is approximately 4 kilograms.
Once $I^N$, $I^O$, and $I^S$ are constructed for the 99th percentile of this transformed variable, we in turn multiply by $-1$ and add $T$ back to restore the interval for the original first percentile to conduct inference about the extremely low birth weights.

\begin{table}
\renewcommand{\arraystretch}{0.85} 
\centering
\begin{tabular}{lc cccccc c cccccccc}
\multicolumn{8}{c}{Tail Index} \\
\hline

Subsample &  & \multicolumn{2}{c}{$H^N$} &\multicolumn{2}{c}{$H^O$} & \multicolumn{2}{c}{$H^S$} \\ 
\hline
benchmark   &  & [0.27 & 0.29] & [0.24 & 0.31] & [0.25 & 0.31] \\ 
high school &  & [0.29 & 0.30] & [0.26 & 0.33] & [0.27 & 0.30] \\ 
girl        &  & [0.25 & 0.27] & [0.23 & 0.29] & [0.23 & 0.28] \\ 
unmarried   &  & [0.29 & 0.31] & [0.26 & 0.34] & [0.27 & 0.32] \\ 
black       &  & [0.24 & 0.29] & [0.21 & 0.32] & [0.21 & 0.33] \\ 
no pre-visit&  & [0.34 & 0.41] & [0.31 & 0.45] & [0.31 & 0.46] \\ 
smoke       &  & [0.22 & 0.27] & [0.20 & 0.30] & [0.19 & 0.29] \\ 
older       &  & [0.29 & 0.31] & [0.26 & 0.34] & [0.26 & 0.32] \\ 
\hline
\\
\multicolumn{8}{c}{First Percentile in Kilogram}\\
\hline
Subsample &  & \multicolumn{2}{c}{$I^N$} &\multicolumn{2}{c}{$I^O$} & \multicolumn{2}{c}{$I^S$} \\ 
\hline
benchmark   &  & [1.46 & 1.56] & [1.30 & 1.72] & [1.35 & 1.68]\\ 
high school &  & [1.48 & 1.52] & [1.31 & 1.69] & [1.43 & 1.68]\\ 
girl        &  & [1.50 & 1.59] & [1.35 & 1.74] & [1.43 & 1.72]\\ 
unmarried   &  & [1.20 & 1.30] & [0.99 & 1.51] & [1.10 & 1.48]\\ 
black       &  & [0.99 & 1.39] & [0.80 & 1.58] & [0.79 & 1.58]\\ 
no pre-visit&  & [0.00 & 0.63] & [0.00 & 1.08] & [0.00 & 1.09]\\ 
smoke       &  & [1.34 & 1.59] & [1.21 & 1.72] & [1.27 & 1.71]\\ 
older       &  & [1.30 & 1.45] & [1.13 & 1.62] & [1.24 & 1.62]\\ 
\hline
\end{tabular}
\caption{The 95\% confidence intervals for the tail index ($H$) and the first percentile ($I$) of the birth weight. 
The superscript $N$ stands for the na\"ive interval in \eqref{eq:naiveCI} and \eqref{eq:quantile_naiveCI}. 
The supersrcipt $O$ stands for the interval in \eqref{eq:honestCI} and \eqref{eq:quantile_honestCI} with $r=1$. 
The superscript $S$ stands for the interval with snooping in \eqref{eq:snoopingCI} and \eqref{eq:quantile_snoopingCI}. 
See the main text for details about the data and the definitions of subsamples.}
\label{tab:natality}
\end{table}

Table \ref{tab:natality} presents the results for the 2002 sample.
The results for other years are similar and hence omitted here to save space.
Key empirical findings from these results can be summarized as follows. 
First, $H^O$ and $H^S$ are similar in length for the tail index, while $I^O$ tends to be slightly longer than $I^S$ for the first percentile.
Second, both of them are substantially longer than the na\"ive intervals, $H^N$ an $I^N$, suggesting that ignoring the bias could lead to misleadingly short intervals.
Third, compared with the benchmark subsample, the mothers who do not have any prenatal visit during pregnancy bear a substantially higher risk of having extremely low infant birth weights. 
This observation remains true even after accounting for a possible bias.

\section{Summary}\label{sec:summary}

In this paper, we present two theoretical results concerning uniform confidence intervals for the tail index and extreme quantiles.
First, we find it impossible to construct a length-optimal confidence interval satisfying the correct uniform coverage over the local non-parametric family of tail distributions.
Second, in light of the impossibility result, we construct an honest confidence interval that is uniformly valid by accounting for the worst-case bias over the local non-parametric class. 
Simulation studies support our theoretical results.
While the na\"ive length-optimal confidence interval suffers from severe under-coverage, our proposed confidence intervals achieve correct coverage.
Applying the proposed method to National Vital Statistics from National Center for Health Statistics, we find that, even after accounting for the worst-case bias bound, having no prenatal visit during pregnancy remain a strong risk factor for low infant birth weight.
This result demonstrates that, despite the impossibility result, it is possible to conduct a robust yet informative statistical inference about the tail index and extreme quantiles.

\bibliographystyle{rss}
\bibliography{mybib}

\newpage
\begin{center}
{\LARGE Supplmentary Appendix \bigskip\\On Uniform Confidence Intervals for the Tail Index and the Extreme Quantile}
\end{center}
\setcounter{page}{1}
\appendix
\section{Proofs}
\subsection{Proof of Theorem \protect\ref{thm:impossibility}}

We need the following two auxiliary lemmas to prove Theorem \ref{thm:impossibility}.
Throughout, suppose that the distributions $F_{n,h}$ and $F_{0}$ are absolutely continuous with their density functions denoted by $f_{n,h}$ and $f_{0}$, respectively.

\begin{lemma}
\label{lemma:half}
We have
\begin{equation*}
\int \left( \frac{f_{n}^{1/2}\left( y\right) }{f_{0}^{1/2}\left( y\right) }-1\right) ^{2}f_{0}\left( y\right) dy=o(1).
\end{equation*}
\end{lemma}

\noindent
\begin{proof}[Proof of Lemma \protect\ref{lemma:half}]
By the definitions of $F_{n,h}^{-1}$ and $F_{0}^{-1}$, we can write 
\begin{align*}
f_{n,h}\left( F_{n,h}^{-1}\left( 1-t\right) \right) =&\frac{t}{\left( \xi_{0}+d_{n}h\left( {t_{n}^{-1}}t\right) \right) F_{n,h}^{-1}\left( 1-t\right) } 
\qquad\text{and} \\
f_{0}\left( F_{n,h}^{-1}\left( 1-t\right) \right) =&\left[ F_{n,h}^{-1}\left( 1-t\right) \right] ^{-1-1/\xi _{0}}/\xi _{0}.
\end{align*}
Hence, it follows that 
\begin{align*}
\frac{f_{0}\left( F_{n,h}^{-1}\left( 1-t\right) \right) }{f_{n,h}\left( F_{n,h}^{-1}\left( 1-t\right) \right) } =&\frac{\left( 1+\xi_{0}^{-1}d_{n}h\left( {t_{n}^{-1}}t\right) \right) }{F_{n,h}^{-1}\left( 1-t\right)
^{1/\xi _{0}}} 
\\
=&\frac{\left( 1+\xi _{0}^{-1}d_{n}h\left( {t_{n}^{-1}}t\right) \right) }{\exp \left( \int_{t}^{1}\frac{d_{n}h\left( {t_{n}^{-1}}v\right) }{v}dv\right) }.
\end{align*}
The change of variables $y=F_{n,h}^{-1}\left( 1-t/{t_{n}^{-1}}\right) $ yields
\begin{align*}
& \text{ \ \ \ }\int \left( \frac{f_{n,h}^{1/2}\left( y\right) }{f_{0}^{1/2}\left( y\right) }-1\right) ^{2}f_{0}\left( y\right) dy 
\\
& =\frac{1}{{t_{n}^{-1}}}\int_{0}^{{t_{n}^{-1}}}\left( \frac{f_{n,h}^{1/2}\left(F_{n,h}^{-1}\left( 1-t/{t_{n}^{-1}}\right) \right) }{f_{0}^{1/2}\left(F_{n,h}^{-1}\left( 1-t/{t_{n}^{-1}}\right) \right) }-1\right) ^{2}\frac{f_{0}\left(F_{n,h}^{-1}\left( 1-t/{t_{n}^{-1}}\right) \right) }{f_{n,h}\left(F_{n,h}^{-1}\left( 1-t/{t_{n}^{-1}}\right) \right) }dt 
\\
& =\frac{1}{{t_{n}^{-1}}}\int_{0}^{{t_{n}^{-1}}}\left[ \left( \exp \left( \frac{d_{n}}{ 2\xi _{0}}\int_{t}^{{t_{n}^{-1}}}\frac{h\left( s\right) }{s}ds\right) -\left(1+d_{n}\frac{h\left( t\right) }{\xi _{0}}\right) ^{1/2}\right) \right] ^{2}
\\
& \text{ \ \ \ \ \ \ \ \ \ \ \ \ \ \ \ \ \ \ \ \ \ \ \ \ \ \ \ \ \ \ }\times
\exp \left( -\frac{d_{n}}{\xi _{0}}\int_{t}^{{t_{n}^{-1}}}\frac{h\left( s\right) }{s }ds\right) dt 
\\
& \overset{(i)}{=}\frac{1}{{t_{n}^{-1}}}\int_{0}^{{t_{n}^{-1}}}\left[ \left( 1-\left( 1+d_{n}\frac{h\left( t\right) }{\xi _{0}}\right) ^{1/2}+\frac{d_{n}}{2\xi _{0}}\int_{t}^{{t_{n}^{-1}}}\frac{h\left( s\right) }{s}ds+o(d_{n})\right) \right]^{2} 
\\
& \text{ \ \ \ \ \ \ \ \ \ \ \ \ \ \ \ \ \ \ \ \ \ \ \ \ \ \ \ \ \ \ }\times
\left( 1-\frac{d_{n}}{\xi _{0}}\int_{t}^{{t_{n}^{-1}}}\frac{h\left( s\right) }{s}ds+o\left( d_{n}\right) \right) dt \\
& \overset{(ii)}{=}\frac{1}{{t_{n}^{-1}}}\int_{0}^{{t_{n}^{-1}}}\left[ \left( d_{n}\frac{h\left( t\right) }{\xi _{0}}+\frac{d_{n}}{2\xi _{0}}\int_{t}^{{t_{n}^{-1}}}\frac{h\left( s\right) }{s}ds+o(d_{n})\right) \right] ^{2}\left( 1+o\left(1\right) \right) dt 
\\
& \overset{(iii)}{=}O\left( d_{n}t_{n}\right) 
\\
& \overset{(iv)}{=}o(1),
\end{align*}
where 
equality (i) follows from $\exp (x)=1+x+o(x)$ as $x\rightarrow 0$, 
equality (ii) follows from $\left( 1+x\right)^{1/2}=1+x+o(x)$ as $x\rightarrow \infty $, 
equality (iii) follows from the assumptions that $h$ is uniformly bounded and square integrable for all  $h\in \mathcal{H}(A,\rho )$,
and 
equality (iv) follows from the fact that $d_{n}t_{n} \approx t_{n}^{\rho +1}$ with $t_{n}\rightarrow 0$ (under Condition \ref{cond:k}) and $\rho >0$.
\end{proof}

\bigskip
\begin{lemma}
\label{lemma:Hellinger}\textit{The Hellinger distance $H^{2}\left(f_{0},f_{n,h}\right) $ between }$f_{0}$ and $f_{n,h}$ satisfies 
\textit{
\ }
\begin{equation}
H^{2}\left( f_{0},f_{n,h}\right) \equiv \int \left( \sqrt{f_{0}\left(y\right) }-\sqrt{f_{n,h}\left( y\right) }\right) ^{2}dy=\frac{\left\vert \left\vert h\right\vert \right\vert ^{2}}{4n\xi _{0}^{2}}\left(1+o(1)\right).  \label{eq:Hellinger}
\end{equation}

\end{lemma}

\noindent
\begin{proof}[Proof of Lemma \protect\ref{lemma:Hellinger}]
First, Proposition 2.1 in \citet{drees2001minimax} yields
\begin{equation}
\left[ n^{1/2}\left( f_{n,h}^{1/2}\left( y\right) -f_{0}^{1/2}\left(y\right) \right) -\frac{1}{2}g_{n,h}\left( y\right) f_{0}^{1/2}\left(y\right) \right] ^{2}dy=o(1),  
\label{eq:L2_diff}
\end{equation}
where
\begin{equation*}
g_{n,h}\left( y\right) \equiv \frac{n^{1/2}d_{n}}{\xi _{0}}\int_{y^{-1/\xi_{0}}}^{\infty }\frac{h\left( {t_{n}^{-1}}s\right) }{s}ds-h\left( {t_{n}^{-1}}y^{-1/\xi_{0}}\right) ,\text{ \ \ }y\geq 1.
\end{equation*}
Note by \citet[][p.286]{drees2001minimax} that $g_{n,h}$ satisfies 
\begin{align}
\int g_{n}\left( x\right) f_{0}\left( x\right) dx =&0  \qquad\text{and} \label{eq:integral_g}
\\
\int g_{n}^{2}\left( x\right) f_{0}\left( x\right) dx \rightarrow &\frac{\left\vert \left\vert h\right\vert \right\vert ^{2}}{\xi _{0}^{2}}.
\label{eq:integral_g2}
\end{align}
Therefore, by expanding the square in (\ref{eq:L2_diff}), the equality (\ref{eq:Hellinger}) follows once we establish 
\begin{equation*}
\int f_{n,h}^{1/2}\left( y\right) f_{0}^{1/2}\left( y\right) g_{n,h}\left(y\right) =o(1).
\end{equation*}
This equality follows as 
\begin{align*}
& \text{ \ \ }\int f_{n}^{1/2}\left( y\right) f_{0}^{1/2}\left( y\right) g_{n,h}\left( y\right) dy \\
& =\int \left( \frac{f_{n}^{1/2}\left( y\right) }{f_{0}^{1/2}\left( y\right)  }-1\right) f_{0}\left( x\right) g_{n,h}\left( y\right) dy 
\\
& \overset{(i)}{\leq }\left( \int \left( \frac{f_{n}^{1/2}\left( y\right) }{f_{0}^{1/2}\left( y\right) }-1\right) ^{2}f_{0}\left( y\right) dy\right)^{1/2}\left( \int g_{n,h}^{2}\left( y\right) f_{0}\left( y\right) dy\right)
^{1/2} 
\\
& \overset{(ii)}{=}o(1)\frac{\left\vert \left\vert h\right\vert \right\vert }{\xi _{0}}(1+o(1))^{1/2} 
 =o(1)\text{,}
\end{align*}
where 
inequality (i) follows by Cauchy Schwarz and  equality (ii) follows from Lemma \ref{lemma:half} and (\ref{eq:integral_g2}).
\end{proof}

\bigskip \noindent
\begin{proof}[Proof of Theorem \protect\ref{thm:impossibility}]
We first use Lemma \ref{lemma:Hellinger} to translate the Hellinger distance
between $f_{0}$ and $f_{n,h}$ into the $L_{1}$-distance between $f_{0}^{n}$
and $f_{n,h}^{n}$. This is done by Equation (17) in \citet{low1997} so that%
\begin{align*}
L_{1}\left( f_{0}^{n},f_{n,1}^{n}\right) \equiv &\int \left\vert
f_{0}^{n}\left( y^{\left( n\right) }\right) -f_{n,1}^{n}\left( y^{\left(
n\right) }\right) \right\vert dy^{\left( n\right) } \\
\leq &2\left( 2-2\exp \left( -\frac{\left\vert \left\vert h\right\vert
\right\vert ^{2}}{8\xi _{0}^{2}}\right) \right) ^{1/2}\left( 1+o(1)\right) \\
\leq &2\left( 2-2\exp \left( -\frac{\varepsilon ^{2}}{8\xi _{0}^{2}}\right)
\right) ^{1/2}\left( 1+o(1)\right) \\
\equiv &C\left( \varepsilon ,\xi _{0}\right) \left( 1+o(1)\right) .
\end{align*}

Next, let $\delta $ be arbitrary. By the definition of $\omega \left(
\varepsilon ,F_{0}^{-1},n\right) $, there exists an $h\in \mathcal{H}_{\rho
}\left( \overline{h}\right) $ and $\left\vert \left\vert h\right\vert \right\vert
\leq \varepsilon $ such that%
\begin{equation*}
\left\vert T_{r}\left(F_{n,h}\right) -T_{r}\left(
F_{0}^{-1}\right) \right\vert \geq \omega \left( \varepsilon
,F_{0}^{-1},n\right) -\delta .
\end{equation*}

Let $F_{n,\lambda h}^{-1}=Q_{n,\overline{k}_{n},F_{n,\lambda h}}$ for a short-handed notation. Also, let $f_{n,\lambda h}$ and $\mathbb{P}_{F_{0}^{n}}$ denote the corresponding density and the probability measure, respectively. Since $H\left( Y^{\left( n\right) }\right) $ has probability of coverage of at least $1-\beta $ uniformly over $f_{n,\lambda h}$, it follows that, for any $\lambda \in \left[ 0,1\right] $, 
\begin{equation*}
\mathbb{P}_{F_{n,\lambda h}^{n}}\left( T_{r}\left( F_{n,\lambda h}^{-1}\right) \in H\left( Y^{\left( n\right) }\right) \right) \geq 1-\beta
\end{equation*}
and
\begin{equation*}
\left\vert \left\vert \lambda h\right\vert \right\vert \leq \lambda \left\vert \left\vert h\right\vert \right\vert \leq \lambda \varepsilon .
\end{equation*}
Then, we have
\begin{align*}
&\mathbb{P}_{F_{0}^{n}}\left( T_{r}\left( F_{n,\lambda h}^{-1}\right) \in H\left( Y^{\left( n\right) }\right) \right) 
\\
=&\mathbb{E}_{F_{0}^{n}}\left[ 1 \left[ T_{r}\left( F_{n,\lambda h}^{-1}\right) \in H\left( Y^{\left(
n\right) }\right) \right] \right] 
\\
=&\mathbb{E}_{F_{n,\lambda h}^{n}}\left[ 1\left[ T_{r}\left( F_{n,\lambda h}^{-1}\right) \in H\left( Y^{\left( n\right) }\right) \right] \left( 1+ \frac{f_{0}^{n}-f_{n,h}^{n}}{f_{n,h}^{n}}\right) \right] 
\\
\geq &1-\beta +\mathbb{E}_{F_{n,\lambda h}^{n}}\left[ \frac{ f_{0}^{n}-f_{n,\lambda h}^{n}}{f_{n,\lambda h}^{n}}\right] 
\\
\geq &1-\beta -\int \left\vert f_{n,\lambda h}^{n}\left( y\right) -f_{0}^{n}\left( y\right) \right\vert dy 
\\
\geq &1-\beta -\lambda C.
\end{align*}
By the same lines of argument as in equations (12)--(15) in \citet{low1997}, we obtain
\begin{align*}
&\mathbb{E}_{F_{0}^{n}}\left[ \mu \left( H\left( Y^{\left( n\right) }\right)\right) \right] 
\\
=&\mathbb{E}_{F_{0}^{n}}\left[ \int_{ \mathbb{R} }1\left[ t\in \mu \left( H\left( y^{\left( n\right) }\right) \right) \right] dt\right] 
\\
\geq &\mathbb{E}_{F_{0}^{n}}\int_{0}^{1}\left( T_{r}\left( F_{n,h}^{-1}\right) -T_{r}\left( F_{0}^{-1}\right) \right) 1\left[ T_{r}\left( F_{n,\lambda h}^{-1}\right) \in \mu \left( H\left( y^{\left( n\right) }\right) \right) \right] d\lambda 
\\
=&\left( T_{r}\left( F_{n,h}^{-1}\right) -T_{r}\left( F_{0}^{-1}\right) \right) \int_{0}^{1}\mathbb{P}_{F_{0}^{n}}\left( T_{r}\left( F_{n,\lambda h}^{-1}\right) \in H\left( Y^{\left( n\right) }\right) \right) d\lambda 
\\
\geq &\left( T_{r}\left( F_{n,h}^{-1}\right) -T_{r}\left( F_{0}^{-1}\right) \right) \int_{0}^{1}\left( 1-\beta -\lambda C\left( \varepsilon ,\xi _{0}\right) \right) d\lambda 
\\
=&\left( \omega \left( \varepsilon ,F_{0}^{-1},n\right) -\delta \right) \left( 1-\beta -\frac{C\left( \varepsilon ,\xi _{0}\right) }{2}\right) .
\end{align*}
The inequality (\ref{eq:lb}) now follows since $\delta $ is set to be arbitrary.
\end{proof}

\subsection{Proof of Theorem \protect\ref{thm:index}}

\begin{proof}
    First, we approximate the empirical tail quantile function $Q_{n,r\bar{k}_{n},F_{n,h}}$ with the partial sum process of the standard exponential random variables. 
Let $\eta _{i}$ for $i=1,2,...$ be i.i.d. standard exponential random variables and $S_{i}\equiv \sum_{j=1}^{i}\eta _{j}$, and $\tilde{Q}_{n,r\bar{k}_{n},F}\equiv F^{-1}\left( 1-S_{\left[ \bar{k}_{n}r\right] +1}/n\right)$, $r\in ( 0,1]$.
Let $\mathbb{P}_{n,h}^{n}$ denote the joint distribution of $n$ i.i.d. draws from the distribution $F_{n,h}$. 
By \citet[][Theorem 5.4.3]{Reiss1989} -- see also \citet[][eq.(5.12)]{drees2001minimax} -- the variational distance between the distribution of $Q_{n,r\bar{k}_{n},F_{n,h}}$ under $\mathbb{P}_{n,h}^{n}$ and the distribution of $\tilde{Q}_{n,r\bar{k}_{n},F}$ vanishes uniformly as $n\rightarrow \infty $, that is,
\begin{equation}
\label{eq:reiss1989}
\left\vert \left\vert \mathcal{L}\left( Q_{n,r\bar{k}_{n},F_{n,h}}|\mathbb{P}_{n,h}^{n}\right) -\mathcal{L}\left( \tilde{Q}_{n,r\bar{k}_{n},F}\right) \right\vert \right\vert =O\left( \bar{k}_{n}/n\right) =o(1),
\end{equation}
which implies that
\begin{equation}
\left\vert \left\vert \mathcal{L}\left( T_{r}\left( Q_{n,r\bar{k}_{n},F_{n,h}}\right) |\mathbb{P}_{n,h}^{n}\right) -\mathcal{L}\left( T_{r}\left( \tilde{Q}_{n,r\bar{k}_{n},F}\right) \right) \right\vert \right\vert =o(1)
\label{eq:weak_conv}
\end{equation}
uniformly for all $r\in \lbrack \underline{r},1]$ and $h\in \mathcal{H}\left(A,\rho \right) $. 
Recall from Section \ref{sec:hill} that $T_{r}\left( z\right) =r^{-1}\int_{0}^{r}z\left( t\right) /z\left( r\right) dt$ characterizes Hill's estimator. Hence it suffices to approximate $\tilde{Q}_{n,r\bar{k}_{n},F}$. 
To this end, we employ a strong approximation of $\tilde{Q}_{n,r\bar{k}_{n},F}$ with $F_{n,h}^{-1}\left( 1-\overline{k}_{n}/n\right) $.
Specifically, using \citet[][eq.(5.13)]{drees2001minimax}, we obtain 
\begin{align}
\sup_{h\in \mathcal{H}\left( A,\rho \right) }\sup_{r\in (0,1]}r^{\xi_{0}+1/2+\varepsilon }\Bigg\vert\frac{\tilde{Q}_{n,r\bar{k}_{n},F}}{
F_{n,h}^{-1}\left( 1-\overline{k}_{n}/n\right) }-\Bigg(r^{-\xi _{0}}-\xi_{0}r^{-\left( \xi _{0}+1\right) }\frac{W\left( \overline{k}_{n}r\right) }{
\overline{k}_{n}}  \notag \\
+d_{n}r^{-\xi _{0}}\int_{r}^{1}\frac{h\left( v\right) }{v}dv\Bigg)\Bigg\vert=o\left( \overline{k}_{n}^{-1/2}\right) \text{ a.s.}
\label{eq:quantile_process}
\end{align}
for all $\varepsilon \in \left( 0,1/2\right) $, where $W\left( \cdot \right) $ denotes the standard Wiener process.

In the second step, we exploit the feature of the functional $T_{r}\left(\cdot \right) $. 
Since $T_{r}\left( \cdot \right) $ is scale invariant such that $T_{r}\left( az\right) =T_{r}\left( z\right) $ for any constant $a$, we have 
\begin{equation*}
T_{r}\left( \frac{\tilde{Q}_{n,r\bar{k}_{n},F}}{F_{n,h}^{-1}\left( 1-\overline{k}_{n}/n\right) }\right) =T_{r}\left( \tilde{Q}_{n,r\bar{k}_{n},F}\right) .
\end{equation*}
Moreover, $T_{r}\left( \cdot \right) $ is Hadamard differentiable at $z_{\xi_{0}}:t\longmapsto t^{-\xi _{0}}$ \citep[cf.][Condition 3]{drees1998smooth} in the sense that
\begin{equation}
\frac{T_{r}\left( z_{\xi _{0}}+d_{n}y_{n}\right) -T_{r}\left( z_{\xi_{0}}\right) }{d_{n}}\rightarrow T_{r}^{\prime }\left( y\right)
\label{eq:hadamard}
\end{equation}
uniformly for all functions $y_{n}$ with $\sup_{t\in (0,1]}t^{\xi _{0}+1/2+\varepsilon }\left\vert y_{n}\left( t\right) \right\vert \leq 1$
and $y_{n}\rightarrow y$. 
To derive the expression of $T_{r}^{\prime }\left(\cdot \right) $, we write  
\begin{align*}
T_{r}\left( z_{\xi _{0}}+d_{n}y_{n}\right) -T_{r}\left( z_{\xi _{0}}\right)  =&r^{-1}\int_{0}^{r}\log \left( \frac{t^{-\xi _{0}}+d_{n}y_{n}\left(
t\right) }{r^{-\xi _{0}}+d_{n}y_{n}\left( r\right) }\frac{r^{-\xi _{0}}}{t^{-\xi _{0}}}\right) dt \\
=&r^{-1}\int_{0}^{r}\log \left( 1+d_{n}x_{n}\left( t\right) \right) dt,
\end{align*}
where 
\begin{equation*}
x_{n}\left( t\right) =\frac{t^{\xi _{0}}y_{n}\left( t\right) -r^{\xi_{0}}y_{n}\left( r\right) }{1+d_{n}r^{\xi _{0}}y_{n}\left( r\right) }\rightarrow t^{\xi _{0}}y\left( t\right) -r^{\xi _{0}}y\left( r\right) .
\end{equation*}
Following the derivation in \citet[][p.104]{drees1998estimate}, we obtain
\begin{equation*}
T_{r}^{\prime }\left( y\right) =r^{-1}\int_{0}^{r}\left( t^{\xi _{0}}y\left( t\right) -r^{\xi _{0}}y\left( r\right) \right) dt\text{. }
\end{equation*}

In the final step, we substitute $z_{\xi _{0}}\left( t\right) =t^{-\xi _{0}}$,  $d_{n}\approx \bar{k}_{n}^{-1/2}$, and
\begin{equation*}
y_{n}\left( t\right) =\xi _{0}t^{-\left( \xi _{0}+1\right) }\frac{W\left( 
\overline{k}_{n}t\right) }{\overline{k}_{n}^{1/2}}+t^{-\xi _{0}}\int_{t}^{1}\frac{h\left( v\right) }{v}dv.
\end{equation*}
By (\ref{eq:quantile_process}), (\ref{eq:hadamard}), and the functional delta method,  we have
\begin{equation}
\sup_{h\in \mathcal{H}_{\rho }}\sup_{r\in (0,1]}\left\vert  \begin{array}{c} T_{r}\left( \tilde{Q}_{n,r\bar{k}_{n},F}\right) -\left( \xi _{0}+\bar{k}_{n}^{-1/2}\xi _{0}T_{r}^{\prime }\left( z_{\xi _{0}+1}W_{n}\right) \right) 
\\ 
+d_{n}T_{r}^{\prime }\left( t^{-\xi _{0}}\int_{t}^{1}\frac{h\left( v\right) }{v}dv\right) 
\end{array}
\right\vert =o\left( \overline{k}_{n}^{-1/2}\right) \text{ a.s.}
\label{eq:xi_approx}
\end{equation}
where $W_{n}:t\longmapsto -\bar{k}_{n}^{-1/2}W\left( \bar{k}_{n}t\right) $. 
Use the definition of $T_{r}^{\prime }\left( \cdot \right) $ to obtain that 
\begin{equation}
T_{r}^{\prime }\left( z_{\xi _{0}+1}W_{n}\right) =\frac{1}{r} \int_{0}^{r}\left( s^{-1}W\left( s\right) -r^{-1}W\left( r\right) \right) ds\equiv \mathbb{G}\left( r\right)   \label{eq:xi_G}
\end{equation}
and 
\begin{align}
T_{r}^{\prime }\left( t^{-\xi _{0}}\int_{t}^{r}\frac{h\left( v\right) }{v} dv\right) -h\left( 0\right)  =&r^{-1}\int_{0}^{r}\left( \int_{s}^{1}\frac{%
h\left( v\right) }{v}dv-\int_{r}^{1}\frac{h\left( v\right) }{v}dv\right) ds-h\left( 0\right)   \notag \\
=&r^{-1}\int_{0}^{r}\left( \int_{s}^{r}\frac{h\left( v\right) }{v}dv\right) ds-h\left( 0\right)   \notag \\
=&r^{-1}\int_{0}^{r}\left( \int_{s}^{r}\frac{h\left( v\right) -h\left( 0\right) }{v}dv\right) ds  \notag \\
\equiv &\mathbb{B}\left( r;h\right) ,  \label{eq:xi_B}
\end{align}
where we used the equality $r^{-1}\int_{0}^{r}\left( \log t-\log r\right) dt=1$. 
Thus, in view of $\xi _{n,h}=\xi _{0}+d_{n}h\left( 0\right) $ and combining (\ref{eq:xi_approx}) to (\ref{eq:xi_B}), we find
\begin{equation*}
T_{r}\left( \tilde{Q}_{n,t\bar{k}_{n},F}\right) -\xi _{n,h}=\bar{k}_{n}^{-1/2}\xi _{0}\mathbb{G}\left( r\right) +d_{n}\mathbb{B}\left( r;h\right) +o_{\text{a.s.}}( \overline{k}_{n}^{-1/2} ),
\end{equation*}
where $o_{\text{a.s.}}(\cdot )$ is uniform for all $r\in \lbrack \underline{r},1]$ and $h\in \mathcal{H}\left( A,\rho \right) $. 
We conclude from (\ref{eq:weak_conv}) that
\begin{align*}
\sup_{h\in \mathcal{H}\left( A,\rho \right) }\sup_{r\in \lbrack \underline{r},1]}\sup_{x\in \mathbb{R} }\left\vert \mathbb{P}_{n,h}^{n}\left( T_{r}\left( {Q}_{n,r\bar{k}_{n},F_{n,h}}\right) -\xi _{n,h}\leq d_{n}x\right) -\mathbb{P}\left( \xi _{0} \mathbb{G}\left( r\right) +\mathbb{B}\left( r;h\right) \leq x\right) \right\vert 
\\
\rightarrow 0.
\end{align*}
This implies that
\begin{equation*}
\overline{k}_{n}^{1/2}\left( T_{r}\left( Q_{n,r\bar{k}_{n},F_{n,h}}\right) -\xi _{n,h}\right) =\xi _{0}\mathbb{G}\left( r\right) +\mathbb{B}\left( r;h\right) +o_{\mathbb{P}_{n,h}^{n}}\left( 1\right) ,
\end{equation*}
where $o_{\mathbb{P}_{n,h}^{n}}\left( 1\right)$ is uniform for all $r\in \lbrack \underline{r},1]$ and $h\in \mathcal{H}\left( A,\rho \right) $. 
This completes the proof.
\end{proof}

\subsection{Proof of Theorem \protect\ref{thm:quantile}}

\noindent
\begin{proof}
We use the notation $g_{n}=\overline k_{n}/\left( np_{n}\right)$.
Condition \ref{cond:p} guarantees that $g_n \rightarrow \infty$. Using the tail quantile process (\ref{eq:quantile_process}) with $t=1$,
we obtain 
\begin{align*}
&\frac{\sqrt{r\overline{k}_n}}{\log \left( g_{n}\right) }\left( \frac{\hat{F}%
_{n,h}^{-1}\left( 1-p_{n}\right) }{F_{n,h}^{-1}\left( 1-p_{n}\right) }%
-1\right) 
\\
=&\frac{\sqrt{r\overline{k}_n}}{\log \left( g_{n}\right) }\left( \frac{%
Y_{n,n- \left\lfloor r\overline k_n \right\rfloor}}{F_{n,h}^{-1}\left( 1-r \overline k_{n}/n\right) }g_{n}^{\hat{\xi}\left(
n,r \overline k_n\right) }\frac{F_{n,h}^{-1}\left( 1- r \overline k_{n}/n\right) }{F_{n,h}^{-1}\left(
1-p_{n}\right) }-1\right) \\
=&g_{n}^{\xi _{n,h}}\frac{F_{n,h}^{-1}\left( 1-r \overline k_{n}/n\right) }{%
F_{n,h}^{-1}\left( 1-p_{n}\right) }\left\{ \frac{\sqrt{r\overline{k}_n}}{\log
\left( g_{n}\right) }\left( g_{n}^{\hat{\xi}\left( n,r \overline k_n\right) -\xi
_{n,h}}-1\right) \right. \\
&+\frac{\sqrt{r\overline{k}_n}}{\log \left( g_{n}\right) }\left( \frac{Y_{n,n- \lfloor r \overline k_n \rfloor }}{%
F_{n,h}^{-1}\left( 1-r \overline k_{n}/n\right) }-1\right) g_{n}^{\hat{\xi}\left(
n,r \overline k_n\right) -\xi _{n,h}} \\
&\left. -\frac{\sqrt{r\overline{k}_n}}{\log \left( g_{n}\right) }\left( \frac{%
F_{n,h}^{-1}\left( 1-p_{n}\right) }{F_{n,h}^{-1}\left( 1-r\overline k_{n}/n\right) }%
g_{n}^{-\xi _{n,h}}-1\right) \right\} \\
\equiv & C_{1n}\left( C_{2n}+C_{3n}-C_{4n}\right) .
\end{align*}%
We aim to establish 
\begin{align}
C_{1n} =&1+o(1),  \label{eq:C1} \\
C_{2n} =&\sqrt{r\overline{k}_n}\left( \hat{\xi}\left( n,r \overline k_n\right) -\xi _{n,h}\right)
+o_{\mathbb{P}^n_{n,h}}(1) ,  \label{eq:C2} \\
C_{3n} =&o_{\mathbb{P}^n_{n,h}}(1), \qquad\text{and}  \label{eq:C3} \\
C_{4n} =&o(1),  \label{eq:C4}
\end{align}
where the $o_{p}\left( 1\right) $ and $o_{\mathbb{P}^n_{n,h}}(1)$ terms are all uniform over both 
$h\in \mathcal{H}_{\rho }\left( A,\rho \right) $ and $r\in \lbrack 
\underline{r},1]$.

First, Condition \ref{cond:k} yields that $r\overline{k}_n/\left(nt_{n}\right) \rightarrow r$. Thus, (\ref{eq:C1}) follows from 
\begin{align*}
C_{1n} & \overset{(i)}{=}\exp \left( \int_{r \overline k_{n}/n}^{1}\frac{d_{n}\left(
h\left( {t_{n}^{-1}}v\right) -h\left( 0\right) \right) }{v}dv-\int_{p_{n}}^{1}\frac{%
d_{n}\left( h\left( {t_{n}^{-1}}v\right) -h\left( 0\right) \right) }{v}dv\right) \\
& =\exp \left( d_{n}\int_{\min \{r \overline k_{n}/n,p_{n}\}}^{\max \{r \overline k_{n}/n,p_{n}\}}%
\frac{h\left( {t_{n}^{-1}}v\right) -h\left( 0\right) }{v}dv\right) \\
& \overset{(ii)}{=}\exp \left( d_{n}\int_{\min \{g_{n},p_{n}{t_{n}^{-1}}\}}^{\max
\{g_{n},p_{n}{t_{n}^{-1}}\}}\frac{h\left( s\right) -h\left( 0\right) }{s}ds\right)
\\
& \leq \exp \left( d_{n}\int_{0}^{1}\frac{\left\vert h\left( s\right)
-h\left( 0\right) \right\vert }{s}ds\right) \\
& \overset{(iii)}{\leq }\exp \left( d_{n}\int_{0}^{1}\frac{As^{\rho }}{s}%
ds\right) \\
& \overset{(iv)}{=}1+O(d_{n}),
\end{align*}%
where 
equality (i) follows from (\ref{eq:local_alternative}), 
equality (ii) follows from a
change of variables, 
inequality (iii) follows from $\left\vert
h\left( s\right) -h\left( 0\right) \right\vert \leq As^{\rho }$ for $h
\in \mathcal{H}\left( A,\rho \right) $, and 
equality (iv) follows from $d_{n}\rightarrow 0$ and $\exp \left(
a_{n}\right) =1+a_{n}+O(a_{n}^{2})$ for any sequence $a_{n}\rightarrow 0$.

Next, consider $C_{2n}$. Note that $( \hat{\xi}\left( n,r \overline k_n\right) -\xi
_{n,h}) \log g_{n} = \sqrt{r \overline k_{n}}( \hat{\xi}\left( n,r \overline k_n\right) -\xi
_{n,h}) \times \log g_{n}/\sqrt{r \overline k_{n}} = o_{\mathbb{P}^n_{n,h}}(1) $ uniformly for all $%
r \in \lbrack \underline{r}, 1]$ and $h\in \mathcal{H}\left( A,\rho
\right) $ by Theorem \ref{thm:index} and Condition \ref{cond:p}. Then, using $%
\exp (a_{n})=1+a_{n}+O(a_{n}^{2})$ for any generic sequence $%
a_{n}\rightarrow 0$, we have that%
\begin{align*}
&\frac{\sqrt{r\overline{k}_n}}{\log \left( g_{n}\right) }\left( g_{n}^{\hat{\xi}%
\left( n,r \overline k_n\right) -\xi _{n,h}}-1\right) \\
=&\sqrt{r \overline k_{n}}\left( \hat{\xi}\left( n,r \overline k_n\right) -\xi _{n,h}\right) \frac{%
\exp \left( \left( \hat{\xi}\left( n,r \overline k_n\right) -\xi _{n,h}\right) \log
g_{n}\right) -1}{\left( \hat{\xi}\left( n,r \overline k_n\right) -\xi _{n,h}\right) \log
g_{n}} \\
=&\sqrt{r \overline k_{n}}\left( \hat{\xi}\left( n,r \overline k_n\right) -\xi _{n,h}\right) \left(
1+O\left( \left( \hat{\xi}\left( n,r \overline k_n\right) -\xi _{n,h}\right) \log
g_{n}\right) \right) \\
=&\sqrt{r \overline k_{n}}\left( \hat{\xi}\left( n,r \overline k_n\right) -\xi _{n,h}\right) \left(
1+o_{\mathbb{P}^n_{n,h}}(1)\right) 
\end{align*}%
uniformly for all $r \in \lbrack \underline{r},1]$ and $h\in \mathcal{%
H}\left( A,\rho \right) $. This yields (\ref{eq:C2}).

Next, consider $C_{3n}$. The equivalence (\ref{eq:reiss1989}) and the tail quantile approximation (\ref{eq:quantile_process}) imply that
\begin{equation*}
\sqrt{r \overline k_{n}}\left( \frac{Y_{n,n-\lfloor r \overline k_n \rfloor }}{F_{n,h}^{-1}\left( 1-r \overline k_{n}/n\right) }%
-1\right) =O_{\mathbb{P}^n_{n,h}}(1)
\end{equation*}%
uniformly  for all $r \in \lbrack \underline{r},%
1]$ and $h\in \mathcal{H}\left( A,\rho \right) $.
The previous step deriving (\ref{eq:C2}) also implies that $g_{n}^{\hat{\xi}%
\left( n,r \overline k_n\right) -\xi _{n,h}}=1+o_{\mathbb{P}^n_{n,h}}(1)$ uniformly for all $r \in \lbrack 
\underline{r},1]$ and $h\in \mathcal{H}\left( A,\rho \right) $.
Therefore, $C_{3n}=\left( \log g_{n}\right) ^{-1}O_{\mathbb{P}^n_{n,h}}(1)\left(
1+o_{\mathbb{P}^n_{n,h}}(1)\right) =o_{\mathbb{P}^n_{n,h}}(1)$ uniformly for all $r\in \lbrack \underline{r},%
1]$ and $h\in \mathcal{H}\left( A,\rho \right) $, implying (\ref%
{eq:C3}).

Finally, consider $C_{4n}$. We have 
\begin{align*}
C_{4n} & \overset{(i)}{=}\frac{\sqrt{r\overline{k}_n}}{\log \left( g_{n}\right) }%
\left( \exp \left( d_{n}\int_{\min \{p_{n},r \overline k_{n}/n\}}^{\max
\{p_{n},r \overline k_{n}/n\}}\frac{\left( h\left( {t_{n}^{-1}}v\right) -h\left( 0\right)
\right) }{v}dv\right) -1\right) \\
& \overset{(ii)}{=}\frac{\sqrt{r\overline{k}_n}}{\log \left( g_{n}\right) }\times
O\left( d_{n}\right) \\
& \overset{(iii)}{=}o(1)
\end{align*}%
uniformly for all $r\in \lbrack \underline{r%
},1]$ and $h\in \mathcal{H}\left( A,\rho \right) $,
where 
equality (i) follows from substituing (\ref{eq:local_alternative}), 
equality (ii) follows
from the derivation of (\ref{eq:C1}), and 
equality (iii) follows from that $%
\sqrt{r \overline k_n} d_{n}=O(1)$ as in Condition \ref{cond:k} and $\log (g_{n})\rightarrow \infty $ as in Condition \ref{cond:p}. The proof is
complete by combining (\ref{eq:C1})--(\ref{eq:C4}).
\end{proof}

\section{Appendix: Choice of $\overline{k}_n$}\label{sec:data_driven_k}
In this appendix section, we present a data-driven choice rule of $\overline{k}_n$ following \citet{guillou2001diagnostic} for completeness and for convenience of readers. 
For other methods of choosing the tail threshold, see, for example, \citet{dreesKaufmann1998select}, \citet{geluk2000}, \citet{danielsson2001using}, and \citet{danielsson2016}. 
See also \citet[][Chapter 4.4]{resnick2007book} for a review. 

We use the shorthand notation $k = \overline{k}_n$ in this section for notational simplicity. 
Define $Z_{i}=i\log (Y_{n:n-i+1}/Y_{n:n-i})$ for $i=1,\ldots ,n$. 
If $Y_{i}$ is exactly Pareto distributed with exponent $1/\xi_0 $, then $Z_{i}$ should be i.i.d. with exponential distribution and $\mathbb{E}\left[ Z_{i}\right] =\xi_0 $. 
Given this observation, we further construct the antisymmetric weights $\{w_{j}\}_{j=1}^{k}$\ such that $w_{j}=-w_{k-j+1}$\ and $\sum_{j=1}^{k}w_{j}=0$.
Then, the statistic
\begin{equation*}
\mathcal{T}_{k}=\left( \sum_{j=1}^{k}w_{j}^{2}\right) ^{-1/2}\hat{\xi}(n,k)^{-1}U_{k}\text{ where } U_{k}=\sum_{j=1}^{k}w_{j}Z_{j}
\end{equation*}
has zero mean and unit variance provided that $Y_{i}$ is exactly Pareto and $\hat{\xi}(n,k)=\xi_0 $.

To evaluate a deviation of the above approximation, we further define the following
criterion based on a moving average of $\mathcal{T}_{k}^{2}$: 
\begin{equation*}
\mathcal{C}_{k}=\left( \left( 2l+1\right) ^{-1}\sum_{j=-l}^{l}\mathcal{T}_{k+j}^{2}\right) ^{1/2},
\end{equation*}
where $l$ equals the integer part of $k/2$. 
Intuitively, the larger $k$ is, the larger bias we have in the Pareto tail approximation, and hence by more $\mathcal{C}_{k}$ exceeds one. 
To obtain an implementable rule, we follow \citet{guillou2001diagnostic} to use $w_{j}=sgn\left( k-2j+1\right) \left\vert k-2j+1\right\vert $ and propose to choose the smallest $k$ that satisfies $\mathcal{C}_{t}>c_{\text{crit}}$ for all $t\geq k$ and some pre-specified constant $c_{\text{crit}}$. 
Again, following \citet{guillou2001diagnostic}, we set $c_{\text{crit}}=1.25$.
For convenience of reference, we write the concrete equation:
\begin{equation}
\hat{k}=\min_{1\leq k\leq n}\{k:\mathcal{C}_{t}>c_{\text{crit}}\text{ for all }t\geq k\}.  \label{k choice}
\end{equation}

\end{document}